\documentclass[11pt]{article}
\usepackage{amsfonts}
\usepackage{amsmath}
\usepackage{amssymb}
\usepackage[a4paper]{geometry}
\usepackage{graphicx}
\usepackage[onehalfspacing]{setspace}

\newtheorem{theorem}{Theorem}

\newtheorem{corollary}[theorem]{Corollary}

\newenvironment{proof}[1][Proof]{\noindent\textbf{#1.} }{\ \rule{0.5em}{0.5em}}
\input{tcilatex}
\begin{document}

\title{A rational trigonometric relationship between the dihedral angles of
a tetrahedron and its circumradius}
\author{Gennady Arshad Notowidigdo \\
%EndAName
School of Mathematics and Statistics\\
UNSW Sydney\\
Sydney,\ NSW, Australia\\
gnotowidigdo@zohomail.com.au}
\date{}
\maketitle

\begin{abstract}
This paper will extend a known relationship between the circumradius and
dihedral angles of a tetrahedron in three-dimensional Euclidean space to
three-dimensional affine space over a general field not of characteristic
two or three, using only the framework of rational trigonometry devised by
Wildberger. In this framework, a linear algebraic view of trigonometry is
presented, which allows the associated three-dimensional vector space of
such a three-dimensional affine space to be equipped with a non-degenerate
symmetric bilinear form. This will also generalise the results presented to
arbitrary geometries parameterised by such a non-degenerate symmetric
bilinear form.
\end{abstract}

\section{Introduction}

The following result, from \cite{Cho}, establishes a relationship between
the dihedral angles of a tetrahedron in three-dimensional Euclidean space
over the real number field and its circumradius.

\begin{theorem}
Let $A_{0}$, $A_{1}$, $A_{2}$ and $A_{3}$ be the points of a tetrahedron in
three-dimensional Euclidean space over the real number field. Let $R$ be its
circumradius, and for distinct $i$ and $j$ in the set $\left\{
0,1,2,3\right\} $, let $\theta _{ij}$ be the interior dihedral angle between 
$A_{i}$ and $A_{j}$. Then, the volume $V$ of the tetrahedron is%
\begin{equation*}
V=\frac{32}{3}\frac{N_{0}N_{1}N_{2}N_{3}}{M^{\frac{3}{2}}}R^{3}
\end{equation*}%
where, for $l$ in the set $\left\{ 0,1,2,3\right\} $,%
\begin{equation*}
N_{l}=\det 
\begin{pmatrix}
1 & \cos \theta _{ij} & \cos \theta _{ik} \\ 
\cos \theta _{ji} & 1 & \cos \theta _{jk} \\ 
\cos \theta _{ki} & \cos \theta _{kj} & 1%
\end{pmatrix}%
\end{equation*}%
and%
\begin{equation*}
M=-\det 
\begin{pmatrix}
0 & \sin ^{2}\theta _{01} & \sin ^{2}\theta _{02} & \sin ^{2}\theta _{03} \\ 
\sin ^{2}\theta _{10} & 0 & \sin ^{2}\theta _{12} & \sin ^{2}\theta _{13} \\ 
\sin ^{2}\theta _{20} & \sin ^{2}\theta _{21} & 0 & \sin ^{2}\theta _{23} \\ 
\sin ^{2}\theta _{30} & \sin ^{2}\theta _{31} & \sin ^{2}\theta _{32} & 0%
\end{pmatrix}%
.
\end{equation*}
\end{theorem}

In this paper, we aim to derive a similar result using only the framework of
rational trigonometry from \cite{WildDP}. In rational trigonometry, the
classical metrical notions of distance and angle are replaced respectively
by quadrance and spread, which have purely linear algebraic definitions.
This allows us to define the metrical notions of rational trigonometry in
three-dimensional affine space; we can thus equip to its associated
three-dimensional vector space a non-degenerate symmetric bilinear form that
can be represented by an invertible symmetric $3\times 3$ matrix, for which
we then have a generalised definition of perpendicularity. This gives us a
general metrical structure by which the results seen in this paper can be
generalised to arbitrary geometries. We should also note that these results
can also be generalised to arbitrary fields not of characteristic $2$ or $3$%
, so that the Zero denominator convention in \cite[p. 28]{WildDP} is adopted
without constant reference.

We will extend a known result from \cite{Crelle} with regards to the
circumradius of a tetrahedron and its volume and side lengths to a general
metrical framework, which allows us to obtain a rational analog of the
aforementioned result. Proving this result also allows us to obtain a result
pertaining to the ratio of the product of opposing dihedral angles to the
product of opposing side lengths, a result found in \cite{NotoWild2} but
verified here in a different way, and to express the circumradius explicitly
in terms of the dihedral angles, volume and areas of the tetrahedron. The
notions of area, volume and dihedral angle from classical trigonometry will
be replaced by the rational trigonometric notions of quadrea, quadrume and
dihedral spread to suit the demands of the paper.

We will also adopt a novel approach to proving the results presented in this
paper, based on the framework of \cite{Noto1}. Here, we take an affine map
from a general tetrahedron to a special tetrahedron whose points are%
\begin{equation*}
X_{0}=\left[ 0,0,0\right] ,\quad X_{1}=\left[ 1,0,0\right] ,\quad X_{2}=%
\left[ 0,1,0\right] \quad \mathrm{and}\quad X_{3}=\left[ 0,0,1\right] .
\end{equation*}%
This special tetrahedron will be named the \emph{Standard tetrahedron}, and
it allows us to analyse a specific tetrahedron over a general symmetric
bilinear form rather than a general tetrahedron over a specific symmetric
bilinear form. Using this tool, a key property of the affine map implies
that any result we prove for the Standard tetrahedron can be generalised to
a general tetrahedron, by way of the inverse affine map; thus, it is
sufficient that any result presented in this paper is proven for the
Standard tetrahedron.

With the use of the Standard tetrahedron, we may be able to prove our
desired results using this powerful mechanism, which makes
computationally-intensive problems more manageable.

\section{Preliminaries}

We start with the three-dimensional affine space over a general field $%
\mathbb{F}$ not of characteristic $2$ or $3$, which we will denote by $%
\mathbb{A}^{3}$. The associated three-dimensional vector space $\mathbb{V}%
^{3}$, which contains three-dimensional row vectors, is then equipped with a
non-degenerate symmetric bilinear form represented by an invertible
symmetric $3\times 3$ matrix $B$ and defined by%
\begin{equation*}
u\cdot _{B}v=uBv^{T}
\end{equation*}%
for vectors $u$ and $v$ in $\mathbb{V}^{3}$, which we will call the $B$\emph{%
-scalar product}\textbf{\ }of $u$ and $v$ (see \cite{NotoWild1}). From this,
we also define the $B$\emph{-quadratic form} by%
\begin{equation*}
Q_{B}\left( v\right) =v\cdot _{B}v
\end{equation*}%
for a vector $v$ in $\mathbb{V}^{3}$. If $u\cdot _{B}v=0$ then $u$ and $v$
are said to be $B$\emph{-perpendicular}. We also say that a vector $v$ is $B$%
\emph{-null} if $Q_{B}\left( v\right) =0$; note that if $B$ is positive
definite, then $v=\mathbf{0}$ is the only $B$-null vector.

The primary objects in $\mathbb{A}^{3}$ are called \emph{points} and are
denoted in this paper as a triple enclosed in rectangular brackets, and the
primary objects in $\mathbb{V}^{3}$ are called \emph{vectors} and are
typically denoted as a three-dimensional row matrix. The association
mentioned above is described by the operation%
\begin{equation*}
X+v=Y
\end{equation*}%
for points $X$ and $Y$ in $\mathbb{A}^{3}$, and $v$ a vector in $\mathbb{V}%
^{3}$, which then allows us to define a vector between $X$ and $Y$ by 
\begin{equation*}
\overrightarrow{XY}=v\sim Y-X.
\end{equation*}%
Here, we denote $Y-X$ to be the affine difference of two points in $\mathbb{A%
}^{3}$, which is equivalent to the vector $v$.

\subsection{Tetrahedron in three-dimensional affine space}

A \emph{tetrahedron} in $\mathbb{A}^{3}$ is a set of four points in $\mathbb{%
A}^{3}$ and typically denoted as $\overline{A_{0}A_{1}A_{2}A_{3}}$. An \emph{%
edge} of a tetrahedron $\overline{A_{0}A_{1}A_{2}A_{3}}$ is then a subset
containing any two distinct points of $\overline{A_{0}A_{1}A_{2}A_{3}}$ and
is denoted by $\overline{A_{i}A_{j}}$ for integers $i$ and $j$ satisfying $%
0\leq i<j\leq 3$. Furthermore, a \emph{triangle} of a tetrahedron $\overline{%
A_{0}A_{1}A_{2}A_{3}}$ is a subset of any three distinct points of $%
\overline{A_{0}A_{1}A_{2}A_{3}}$ and is denoted by $\overline{A_{i}A_{j}A_{k}%
}$ for integers $i$, $j$ and $k$ satisfying $0\leq i<j<k\leq 3$. Note that
there are six edges and four triangles associated to any tetrahedron in $%
\mathbb{A}^{3}$, and that there are three edges associated to each triangle
of such a tetrahedron. We will also define the \emph{midpoint} of the edge $%
\overline{A_{i}A_{j}}$ to be the point $M_{ij}$ satisfying%
\begin{equation*}
\overrightarrow{A_{i}M_{ij}}=\overrightarrow{M_{ij}A_{j}}=\frac{1}{2}%
\overrightarrow{A_{i}A_{j}}.
\end{equation*}

Associated to each edge of a tetrahedron $\overline{A_{0}A_{1}A_{2}A_{3}}$
is a $B$\emph{-quadrance}, which is the number%
\begin{equation*}
Q_{B}\left( \overline{A_{i}A_{j}}\right) =Q_{B}\left( \overrightarrow{%
A_{i}A_{j}}\right) =\overrightarrow{A_{i}A_{j}}\cdot _{B}\overrightarrow{%
A_{i}A_{j}}
\end{equation*}%
for integers $i$ and $j$ satisfying $0\leq i<j\leq 3$. This will be denoted
for the rest of this paper by $Q_{ij}$.

Given a triangle $\overline{A_{i}A_{j}A_{k}}$ of a tetrahedron $\overline{%
A_{0}A_{1}A_{2}A_{3}}$, for integers $i$, $j$ and $k$ satisfying $0\leq
i<j<k\leq 3$, we have the three edges $\overline{A_{i}A_{j}}$, $\overline{%
A_{i}A_{k}}$ and $\overline{A_{j}A_{k}}$ associated to it, with respective $%
B $-quadrances $Q_{ij}$, $Q_{ik}$ and $Q_{jk}$. This allows to associate to $%
\overline{A_{i}A_{j}A_{k}}$ the number%
\begin{equation*}
\mathcal{A}_{B}\left( \overline{A_{i}A_{j}A_{k}}\right) =A\left(
Q_{ij},Q_{ik},Q_{jk}\right)
\end{equation*}%
where%
\begin{equation*}
A\left( a,b,c\right) =\left( a+b+c\right) ^{2}-2\left(
a^{2}+b^{2}+c^{2}\right)
\end{equation*}%
is \emph{Archimedes's function} (see \cite[p. 64]{WildDP}). This quantity is
called the $B$\emph{-quadrea} and will be denoted for the rest of this paper
by $\mathcal{A}_{ijk}$.

Associated to a tetrahedron $\overline{A_{0}A_{1}A_{2}A_{3}}$ itself is the $%
B$\emph{-quadrume}, which is the number%
\begin{equation*}
\mathcal{V}_{B}\left( \overline{A_{0}A_{1}A_{2}A_{3}}\right) =\frac{1}{2}%
E\left( Q_{23},Q_{13},Q_{12},Q_{01},Q_{02,}Q_{03}\right)
\end{equation*}%
where%
\begin{equation*}
E\left( q_{1},q_{2},q_{3},p_{1},p_{2},p_{3}\right) =\det 
\begin{pmatrix}
2p_{1} & p_{1}+p_{2}-q_{3} & p_{1}+p_{3}-q_{2} \\ 
p_{1}+p_{2}-q_{3} & 2p_{2} & p_{2}+p_{3}-q_{1} \\ 
p_{1}+p_{3}-q_{2} & p_{2}+p_{3}-q_{1} & 2p_{3}%
\end{pmatrix}%
\end{equation*}%
is \emph{Euler's four-point function} (see \cite[p. 191]{WildDP}). This
function is essentially the \emph{Cayley-Menger determinant} (see \cite%
{Audet}, \cite{Dorrie} and \cite{Sommerville}) and it satisfies the
properties%
\begin{equation*}
E\left( q_{1},q_{2},q_{3},p_{1},p_{2},p_{3}\right) =E\left(
p_{1},p_{2},p_{3},q_{1},q_{2},q_{3}\right)
\end{equation*}%
and%
\begin{equation*}
E\left( q_{1},q_{2},q_{3},p_{1},p_{2},p_{3}\right) =E\left(
q_{i},q_{j},q_{k},p_{i},p_{j},p_{k}\right)
\end{equation*}%
for any permutation $i$, $j$ and $k$ of the integers $1$, $2$ and $3$. For
the rest of this paper, this quantity is denoted by $\mathcal{V}$.

For $0\leq i<j\leq 3$ and indices $k$ and $l$ distinct from $i$ and $j$, we
can associate to a pair of triangles $\overline{A_{i}A_{j}A_{k}}$ and $%
\overline{A_{i}A_{j}A_{l}}$ of a tetrahedron $\overline{A_{0}A_{1}A_{2}A_{3}}
$ the number%
\begin{equation*}
E_{ij}=\frac{4Q_{ij}\mathcal{V}}{\mathcal{A}_{ijk}\mathcal{A}_{ijl}}.
\end{equation*}%
This quantity will be called the $B$\emph{-dihedral spread} (see \cite%
{NotoWild2})\ between the triangles $\overline{A_{i}A_{j}A_{k}}$ and $%
\overline{A_{i}A_{j}A_{l}}$, with the edge $\overline{A_{i}A_{j}}$ being
common in these two triangles. Here, the result of the Tetrahedron dihedral
spread formula in \cite{NotoWild2} is used as a definition to simplify our
discussion.

\subsection{Standard tetrahedron}

Consider an affine map which sends a general tetrahedron $\overline{%
A_{0}A_{1}A_{2}A_{3}}$ to the tetrahedron $\overline{X_{0}X_{1}X_{2}X_{3}}$,
where%
\begin{equation*}
X_{0}=\left[ 0,0,0\right] ,\quad X_{1}=\left[ 1,0,0\right] ,\quad X_{2}=%
\left[ 0,1,0\right] \quad \mathrm{and}\quad X_{3}=\left[ 0,0,1\right] .
\end{equation*}%
Such a tetrahedron will be called the \emph{Standard tetrahedron} (see \cite%
{Noto1}). If we have a $C$-scalar product on $\mathbb{V}^{3}$, the affine
map induces a new scalar product given by%
\begin{eqnarray*}
u\cdot _{C}v &=&uCv^{T}=u\left( LL^{-1}\right) C\left( LL^{-1}\right)
^{T}v^{T} \\
&=&\left( uL\right) \left[ \left( L^{-1}\right) C\left( L^{-1}\right) ^{T}%
\right] \left( vL\right) ^{T}
\end{eqnarray*}%
where $L$ is the matrix representing the linear component of the affine map.
For $M=L^{-1}$, we set the matrix $MCM^{T}$ to be the matrix $B$, so that%
\begin{equation*}
u\cdot _{C}v=\left( uL\right) \cdot _{B}\left( vL\right) .
\end{equation*}%
With this tool we may prove a result for a general tetrahedron by verifying
it for the Standard tetrahedron without any loss of generality, due to the
preservation of various geometric objects under an affine map.

\subsubsection{Trigonometric quantities of the Standard tetrahedron}

In what follows, let%
\begin{equation*}
B=%
\begin{pmatrix}
a_{1} & b_{3} & b_{2} \\ 
b_{3} & a_{2} & b_{1} \\ 
b_{2} & b_{1} & a_{3}%
\end{pmatrix}%
\end{equation*}%
we define%
\begin{equation*}
r_{1}=a_{2}+a_{3}-2b_{1},\quad r_{2}=a_{1}+a_{3}-2b_{2},\quad
r_{3}=a_{1}+a_{2}-2b_{3}
\end{equation*}%
and%
\begin{equation*}
\Delta =\det B=a_{1}a_{2}\allowbreak
a_{3}+2b_{1}b_{2}b_{3}-a_{1}b_{1}^{2}-a_{2}b_{2}^{2}-a_{3}b_{3}^{2}.
\end{equation*}%
We will also define%
\begin{equation*}
\limfunc{adj}B=%
\begin{pmatrix}
a_{2}a_{3}-b_{1}^{2} & b_{1}b_{2}-a_{3}b_{3} & b_{1}b_{3}-a_{2}b_{2} \\ 
b_{1}b_{2}-a_{3}b_{3} & a_{1}a_{3}-b_{2}^{2} & b_{2}b_{3}-a_{1}b_{1} \\ 
b_{1}b_{3}-a_{2}b_{2} & b_{2}b_{3}-a_{1}b_{1} & a_{1}a_{2}-b_{3}^{2}%
\end{pmatrix}%
=%
\begin{pmatrix}
\alpha _{1} & \beta _{3} & \beta _{2} \\ 
\beta _{3} & \alpha _{2} & \beta _{1} \\ 
\beta _{2} & \beta _{1} & \alpha _{3}%
\end{pmatrix}%
\end{equation*}%
to be the \emph{adjoint matrix} (see \cite{AR}) of $B$, so that we may define%
\begin{equation*}
D=\alpha _{1}+\alpha _{2}+\alpha _{3}+2\beta _{1}+2\beta _{2}+2\beta _{3}.
\end{equation*}%
The $B$-quadrances of $\overline{X_{0}X_{1}X_{2}X_{3}}$ are by definition%
\begin{equation*}
Q_{01}=%
\begin{pmatrix}
1 & 0 & 0%
\end{pmatrix}%
\begin{pmatrix}
a_{1} & b_{3} & b_{2} \\ 
b_{3} & a_{2} & b_{1} \\ 
b_{2} & b_{1} & a_{3}%
\end{pmatrix}%
\begin{pmatrix}
1 & 0 & 0%
\end{pmatrix}%
^{T}=a_{1}
\end{equation*}%
and similarly%
\begin{equation*}
Q_{02}=a_{2},\quad Q_{03}=a_{3},\quad Q_{23}=r_{1},\quad Q_{13}=r_{2}\quad 
\mathrm{and}\quad Q_{12}=r_{3}.
\end{equation*}%
The $B$-quadreas of $\overline{X_{0}X_{1}X_{2}X_{3}}$ are by definition%
\begin{eqnarray*}
\mathcal{A}_{012} &=&A\left( Q_{01},Q_{02},Q_{12}\right) =\left(
a_{1}+a_{2}+r_{3}\right) ^{2}-2\left( a_{1}^{2}+a_{2}^{2}+r_{3}^{2}\right) \\
&=&4\left( a_{1}a_{2}-b_{3}^{2}\right) =4\alpha _{3}
\end{eqnarray*}%
and similarly%
\begin{equation*}
\mathcal{A}_{013}=4\alpha _{2},\quad \mathcal{A}_{023}=4\alpha _{1}\quad 
\mathrm{and\quad }\mathcal{A}_{123}=4D
\end{equation*}%
and, by definition, the $B$-quadrume of $\overline{X_{0}X_{1}X_{2}X_{3}}$ 
\begin{eqnarray*}
\mathcal{V} &=&\frac{1}{2}\det 
\begin{pmatrix}
2Q_{01} & Q_{01}+Q_{02}-Q_{12} & Q_{01}+Q_{03}-Q_{13} \\ 
Q_{01}+Q_{02}-Q_{12} & 2Q_{02} & Q_{02}+Q_{03}-Q_{23} \\ 
Q_{01}+Q_{03}-Q_{13} & Q_{02}+Q_{03}-Q_{23} & 2Q_{03}%
\end{pmatrix}
\\
&=&\frac{1}{2}\det 
\begin{pmatrix}
2a_{1} & 2b_{3} & 2b_{2} \\ 
2b_{3} & 2a_{2} & 2b_{1} \\ 
2b_{2} & 2b_{1} & 2a_{3}%
\end{pmatrix}%
=4\Delta .
\end{eqnarray*}%
Finally, the $B$-dihedral spreads of $\overline{X_{0}X_{1}X_{2}X_{3}}$ are
by definition%
\begin{equation*}
E_{01}=\frac{4Q_{01}\mathcal{V}}{\mathcal{A}_{012}\mathcal{A}_{013}}=\frac{%
4a_{1}\left( 4\Delta \right) }{\left( 4\alpha _{3}\right) \left( 4\alpha
_{2}\right) }=\frac{a_{1}\Delta }{\alpha _{2}\alpha _{3}}
\end{equation*}%
and similarly 
\begin{equation*}
E_{02}=\frac{a_{2}\Delta }{\alpha _{1}\alpha _{3}},\quad E_{03}=\frac{%
a_{3}\Delta }{\alpha _{1}\alpha _{2}},\quad E_{23}=\frac{r_{1}\Delta }{%
\alpha _{1}D},\quad E_{13}=\frac{r_{2}\Delta }{\alpha _{2}D}\quad \mathrm{and%
}\text{\quad }E_{12}=\frac{r_{3}\Delta }{\alpha _{3}D}.
\end{equation*}

\subsection{Circumquadrance of tetrahedron}

One of the most important centres of a tetrahedron is its \emph{circumcentre}
(see \cite[pp. 82-83]{Narayan}). In a general metrical framework, the
circumcentre will end up being dependent on the choice of non-degenerate
symmetric bilinear form. With this in mind, we proceed to find a suitable
generalisation.

We start by defining a \emph{plane} through a point $A$ and $B$%
-perpendicular to a vector $v$ to be the space of points $X$ satisfying the
equation%
\begin{equation*}
v\cdot _{B}\overrightarrow{AX}=0.
\end{equation*}%
Then we may define a $B$\emph{-midplane} associated to an edge $\overline{%
A_{i}A_{j}}$ of a tetrahedron $\overline{A_{0}A_{1}A_{2}A_{3}}$ to be a
plane through the midpoint $M_{ij}$ and $B$-perpendicular to the vector $%
\overrightarrow{A_{i}A_{j}}$, for $0\leq i<j\leq 3$. There are six $B$%
-midplanes in total for a general tetrahedron. The following result
establishes the concurrency of each of the six $B$-midplanes of a
tetrahedron.

\begin{theorem}[Tetrahedron $B$-circumcentre theorem]
The six $B$-midplanes associated to each edge of a tetrahedron $\overline{%
A_{0}A_{1}A_{2}A_{3}}$ meet at a single point.
\end{theorem}

\begin{proof}
Without loss of generality, transform $\overline{A_{0}A_{1}A_{2}A_{3}}$ to
the Standard tetrahedron $\overline{X_{0}X_{1}X_{2}X_{3}}$. Note that this
will induce a new non-degenerate symmetric bilinear form, but since we
started with a general symmetric bilinear form we may use the same one
without any loss of generality. For $0\leq i<j\leq 3$, let $M_{ij}$ be the
midpoint of $\overline{X_{i}X_{j}}$, so that any point $C=\left[ x,y,z\right]
$ on each $B$-midplane of $\overline{X_{0}X_{1}X_{2}X_{3}}$ satisfies the
equation%
\begin{equation*}
\overrightarrow{X_{i}X_{j}}\cdot _{B}\overrightarrow{M_{ij}C}=0.
\end{equation*}%
This yields six equations, namely%
\begin{equation*}
a_{1}x+b_{3}y+b_{2}z=\frac{1}{2}a_{1},
\end{equation*}%
\begin{equation*}
b_{3}x+a_{2}y+b_{1}z=\frac{1}{2}a_{2},
\end{equation*}%
\begin{equation*}
b_{2}x+b_{1}y+a_{3}z=\frac{1}{2}a_{3},
\end{equation*}%
\begin{equation*}
\left( b_{3}-a_{1}\right) x+\left( a_{2}-b_{3}\right) y+\left(
b_{1}-b_{2}\right) z=\frac{1}{2}\left( a_{2}-a_{1}\right) \allowbreak ,
\end{equation*}%
\begin{equation*}
\allowbreak \left( b_{2}-a_{1}\right) x+\left( b_{1}-b_{3}\right) y+\left(
a_{3}-b_{2}\right) z=\frac{1}{2}\left( a_{3}-a_{1}\right) \allowbreak
\end{equation*}%
and%
\begin{equation*}
\left( b_{2}-b_{3}\right) x+\left( b_{1}-a_{2}\right) y+\left(
a_{3}-b_{1}\right) z=\frac{1}{2}\left( a_{3}-a_{2}\right) \allowbreak .
\end{equation*}%
From the first three equations,%
\begin{equation*}
C=\left[ \frac{\alpha _{1}a_{1}+\beta _{3}a_{2}+\beta _{2}a_{3}}{2\Delta },%
\frac{\beta _{3}a_{1}+\alpha _{2}a_{2}+\beta _{1}a_{3}}{2\Delta },\frac{%
\beta _{2}a_{1}+\beta _{1}a_{2}+\alpha _{3}a_{3}}{2\Delta }\right] .
\end{equation*}%
Substitute the co-ordinates of this point into the last three equations to
get the desired result.
\end{proof}

The intersection point obtained from the proof of the Tetrahedron
circumcentre theorem will be called the $B$\emph{-circumcentre} of the
Standard tetrahedron $\overline{X_{0}X_{1}X_{2}X_{3}}$; to obtain the $B$%
-circumcentre of a general tetrahedron, we merely perform the inverse affine
map on such a point.

The $B$-quadrance between the $B$-circumcentre of a tetrahedron and any
point of a tetrahedron is called the $B$\emph{-circumquadrance}, and will be
denoted by $R$. The $B$-circumquadrance of the Standard tetrahedron $%
\overline{X_{0}X_{1}X_{2}X_{3}}$ is then%
\begin{equation*}
R=Q_{B}\left( \overrightarrow{X_{0}C}\right) =\frac{A\left(
a_{1}b_{1},a_{2}b_{2},a_{3}b_{3}\right) +a_{1}a_{2}a_{3}\left( D-4\left(
b_{1}+b_{2}+b_{3}\right) \right) }{4\Delta }
\end{equation*}

The following result, which is an extension of Crelle's result in \cite%
{Crelle} from the Euclidean setting to a more general setting, links the $B$%
-circumquadrance of a tetrahedron with its $B$-quadrume and $B$-quadrances.

\begin{theorem}[Crelle's circumquadrance formula]
For a tetrahedron $\overline{A_{0}A_{1}A_{2}A_{3}}$ with $B$-quadrances $%
Q_{ij}$, for $0\leq i<j\leq 3$, $B$-quadrume $\mathcal{V}$ and $B$%
-circumquadrance $R$, the relation%
\begin{equation*}
4\mathcal{V}R=A\left( Q_{01}Q_{23},Q_{02}Q_{13},Q_{03}Q_{12}\right)
\end{equation*}%
is satisfied.
\end{theorem}

\begin{proof}
Without loss of generality, transform $\overline{A_{0}A_{1}A_{2}A_{3}}$ to
the Standard tetrahedron $\overline{X_{0}X_{1}X_{2}X_{3}}$; it is sufficient
to prove the required result for this tetrahedron. So,%
\begin{eqnarray*}
R &=&\frac{A\left( a_{1}b_{1},a_{2}b_{2},a_{3}b_{3}\right)
+a_{1}a_{2}a_{3}\left( a_{1}+a_{2}+a_{3}-2b_{1}-2b_{2}-2b_{3}\right) }{%
4\Delta } \\
&=&\frac{A\left( a_{1}r_{1},a_{2}r_{2},a_{3}r_{3}\right) }{16\Delta }=\frac{%
A\left( Q_{01}Q_{23},Q_{02}Q_{13},Q_{03}Q_{12}\right) }{4\mathcal{V}}.
\end{eqnarray*}%
The required result follows.
\end{proof}

Crelle's circumquadrance formula allows us to conveniently write the $B$%
-circumquadrance of $\overline{X_{0}X_{1}X_{2}X_{3}}$ as%
\begin{equation*}
R=\frac{A\left( a_{1}r_{1},a_{2}r_{2},a_{3}r_{3}\right) }{16\Delta }.
\end{equation*}

\section{Main result}

We now present the main result of this paper, which generalises Theorem 1 to
the rational trigonometric setting over an arbitrary symmetric bilinear form.

\begin{theorem}[Circumquadrance dihedral spread theorem]
For a tetrahedron $\overline{A_{0}A_{1}A_{2}A_{3}}$ in $\mathbb{A}^{3}$, let 
$\mathcal{V}$ be its $B$-quadrume, $\mathcal{A}_{ijk}$ be the $B$-quadrea of
the triangle $\overline{A_{i}A_{j}A_{k}}$, $E_{ij}$ be the $B$-dihedral
spread between $\overline{A_{i}A_{j}A_{k}}$ and $\overline{A_{i}A_{j}A_{l}}$%
, and $R$ be its $B$-circumquadrance. Then,%
\begin{equation*}
\left( \mathcal{A}_{012}\mathcal{A}_{013}\mathcal{A}_{023}\mathcal{A}%
_{123}\right) ^{2}M=1024\mathcal{V}^{5}R,
\end{equation*}%
where%
\begin{equation*}
M=-\det 
\begin{pmatrix}
0 & E_{01} & E_{02} & E_{03} \\ 
E_{01} & 0 & E_{12} & E_{13} \\ 
E_{02} & E_{12} & 0 & E_{23} \\ 
E_{03} & E_{13} & E_{23} & 0%
\end{pmatrix}%
=A\left( E_{01}E_{23},E_{02}E_{13},E_{03}E_{12}\right) .
\end{equation*}
\end{theorem}

\begin{proof}
Perform an affine map on $\overline{A_{0}A_{1}A_{2}A_{3}}$ to the Standard
tetrahedron $\overline{X_{0}X_{1}X_{2}X_{3}}$, so that we only require to
prove this result on $\overline{X_{0}X_{1}X_{2}X_{3}}$. We then have%
\begin{equation*}
E_{01}E_{23}=\frac{a_{1}r_{1}\Delta ^{2}}{\alpha _{1}\alpha _{2}\alpha _{3}D}%
,\quad E_{02}E_{13}=\frac{a_{2}r_{2}\Delta ^{2}}{\alpha _{1}\alpha
_{2}\alpha _{3}D},\quad E_{03}E_{12}=\frac{a_{3}r_{3}\Delta ^{2}}{\alpha
_{1}\alpha _{2}\alpha _{3}D}
\end{equation*}%
and%
\begin{equation*}
\mathcal{A}_{012}\mathcal{A}_{013}\mathcal{A}_{023}\mathcal{A}%
_{123}=256\alpha _{1}\alpha _{2}\alpha _{3}D
\end{equation*}%
so that%
\begin{eqnarray*}
M &=&A\left( E_{01}E_{23},E_{02}E_{13},E_{03}E_{12}\right) \\
&=&\left( \frac{\Delta ^{2}\left( a_{1}r_{1}+a_{2}r_{2}+a_{3}r_{3}\right) }{%
\alpha _{1}\alpha _{2}\alpha _{3}D}\right) ^{2}-2\left( \frac{\Delta
^{4}\left( a_{1}^{2}r_{1}^{2}+a_{2}^{2}r_{2}^{2}+a_{3}^{2}r_{3}^{2}\right) }{%
\left( \alpha _{1}\alpha _{2}\alpha _{3}D\right) ^{2}}\right) \\
&=&\frac{\Delta ^{4}\left( a_{1}r_{1}+a_{2}r_{2}+a_{3}r_{3}\right) ^{2}}{%
\left( \alpha _{1}\alpha _{2}\alpha _{3}D\right) ^{2}}-\frac{2\Delta
^{4}\left( a_{1}^{2}r_{1}^{2}+a_{2}^{2}r_{2}^{2}+a_{3}^{2}r_{3}^{2}\right) }{%
\left( \alpha _{1}\alpha _{2}\alpha _{3}D\right) ^{2}} \\
&=&\frac{\Delta ^{4}A\left( a_{1}r_{1},a_{2}r_{2},a_{3}r_{3}\right) }{\left(
\alpha _{1}\alpha _{2}\alpha _{3}D\right) ^{2}}
\end{eqnarray*}%
and thus%
\begin{eqnarray*}
\left( \mathcal{A}_{012}\mathcal{A}_{013}\mathcal{A}_{023}\mathcal{A}%
_{123}\right) ^{2}M &=&\frac{\Delta ^{4}A\left(
a_{1}r_{1},a_{2}r_{2},a_{3}r_{3}\right) }{\left( \alpha _{1}\alpha
_{2}\alpha _{3}D\right) ^{2}}\left( 256\alpha _{1}\alpha _{2}\alpha
_{3}D\right) ^{2} \\
&=&65536\Delta ^{4}A\left( a_{1}r_{1},a_{2}r_{2},a_{3}r_{3}\right) .
\end{eqnarray*}%
By Crelle's circumquadrance formula,%
\begin{equation*}
\left( \mathcal{A}_{012}\mathcal{A}_{013}\mathcal{A}_{023}\mathcal{A}%
_{123}\right) ^{2}M=1024R\left( 1024\Delta ^{5}\right) =1024\left( 4\Delta
\right) ^{5}R=1024\mathcal{V}^{5}R
\end{equation*}%
as required.
\end{proof}

We can now provide an alternative expression for the Circumquadrance
dihedral spread theorem.

\begin{corollary}
If $N=A\left( Q_{01}Q_{23},Q_{02}Q_{13},Q_{03}Q_{12}\right) $, where, for $%
0\leq i<j\leq 3$, $Q_{ij}$ denotes the $B$-quadrances of a tetrahedron $%
\overline{A_{0}A_{1}A_{2}A_{3}}$, then the Circumquadrance dihedral spread
theorem can alternatively be expressed as%
\begin{equation*}
\left( \mathcal{A}_{012}\mathcal{A}_{013}\mathcal{A}_{023}\mathcal{A}%
_{123}\right) ^{2}M=256\mathcal{V}^{4}N.
\end{equation*}
\end{corollary}

\begin{proof}
From Crelle's circumquadrance formula, we have that%
\begin{equation*}
R=\frac{N}{4\mathcal{V}}.
\end{equation*}%
Substitute into the Circumquadrance dihedral spread theorem to obtain%
\begin{equation*}
\left( \mathcal{A}_{012}\mathcal{A}_{013}\mathcal{A}_{023}\mathcal{A}%
_{123}\right) ^{2}M=1024\mathcal{V}^{5}\left( \frac{N}{4\mathcal{V}}\right)
=256\mathcal{V}^{4}N,
\end{equation*}%
as required.
\end{proof}

We can now derive a relationship between $M$ and $N$, which originates from 
\cite{NotoWild2}, we use the tools presented in this paper to provide an
alternate proof.

\begin{theorem}[Dihedral spread ratio theorem]
Given a tetrahedron $\overline{A_{0}A_{1}A_{2}A_{3}}$ in $\mathbb{A}^{3}$
with $B$-dihedral spreads $E_{ij}$ and $B$-quadrances $Q_{ij}$, where $0\leq
i<j\leq 3$, the relation%
\begin{equation*}
\frac{E_{01}E_{23}}{Q_{01}Q_{23}}=\frac{E_{02}E_{13}}{Q_{02}Q_{13}}=\frac{%
E_{03}E_{12}}{Q_{03}Q_{12}}
\end{equation*}%
is satisfied.
\end{theorem}

\begin{proof}
We start with the fact that%
\begin{equation*}
A\left( \lambda a,\lambda b,\lambda c\right) =\lambda ^{2}A\left(
a,b,c\right)
\end{equation*}%
which can be easily verified by the reader. Now, let $\mathcal{V}$ be the $B$%
-quadrume of $\overline{A_{0}A_{1}A_{2}A_{3}}$ and let $\mathcal{A}_{ijk}$
be their $B$-quadreas, for $0\leq i<j<k\leq 3$. Then rearrange the equation
of Corollary 5 to get%
\begin{equation*}
M=\frac{256\mathcal{V}^{4}}{\left( \mathcal{A}_{012}\mathcal{A}_{013}%
\mathcal{A}_{023}\mathcal{A}_{123}\right) ^{2}}N=\left( \frac{16\mathcal{V}%
^{2}}{\mathcal{A}_{012}\mathcal{A}_{013}\mathcal{A}_{023}\mathcal{A}_{123}}%
\right) ^{2}N
\end{equation*}%
for%
\begin{equation*}
M=A\left( E_{01}E_{23},E_{02}E_{13},E_{03}E_{12}\right) \quad \mathrm{%
and\quad }N=A\left( Q_{01}Q_{23},Q_{02}Q_{13},Q_{03}Q_{12}\right) .
\end{equation*}%
So,%
\begin{equation*}
M=A\left( \frac{16\mathcal{V}^{2}Q_{01}Q_{23}}{\mathcal{A}_{012}\mathcal{A}%
_{013}\mathcal{A}_{023}\mathcal{A}_{123}},\frac{16\mathcal{V}^{2}Q_{02}Q_{13}%
}{\mathcal{A}_{012}\mathcal{A}_{013}\mathcal{A}_{023}\mathcal{A}_{123}},%
\frac{16\mathcal{V}^{2}Q_{03}Q_{12}}{\mathcal{A}_{012}\mathcal{A}_{013}%
\mathcal{A}_{023}\mathcal{A}_{123}}\right) .
\end{equation*}%
and thus by comparison%
\begin{equation*}
\frac{E_{01}E_{23}}{Q_{01}Q_{23}}=\frac{E_{02}E_{13}}{Q_{02}Q_{13}}=\frac{%
E_{03}E_{12}}{Q_{03}Q_{12}}=\frac{16\mathcal{V}^{2}}{\mathcal{A}_{012}%
\mathcal{A}_{013}\mathcal{A}_{023}\mathcal{A}_{123}}
\end{equation*}%
as required.
\end{proof}

From the proof of the Dihedral spread ratio theorem, we can define%
\begin{equation*}
K=\frac{16\mathcal{V}^{2}}{\mathcal{A}_{012}\mathcal{A}_{013}\mathcal{A}%
_{023}\mathcal{A}_{123}}
\end{equation*}%
which in \cite{NotoWild2} is called the \emph{Richardson number} of a
tetrahedron; this quantity originates from \cite{Richardson} and its
geometric meaning is yet to be fully understood (though the author of the
paper uses it frequently). From the Circumquadrance dihedral spread theorem,
we see that 
\begin{equation*}
R=\frac{\left( \mathcal{A}_{012}\mathcal{A}_{013}\mathcal{A}_{023}\mathcal{A}%
_{123}\right) ^{2}M}{1024\mathcal{V}^{5}}=\frac{M}{4\mathcal{V}}\left( \frac{%
\mathcal{A}_{012}\mathcal{A}_{013}\mathcal{A}_{023}\mathcal{A}_{123}}{16%
\mathcal{V}^{2}}\right) ^{2}=\frac{M}{4\mathcal{V}K^{2}}.
\end{equation*}%
Crelle's circumquadrance formula is immediate from this, as $M=K^{2}N$ from
the proof of the Dihedral spread ratio theorem.

\section*{Acknowledgement}

This paper is the result of independent research conducted by the author.
However, the author would like to thank Prof. Norman Wildberger of UNSW
Sydney in Sydney, NSW, Australia for the inspiration of this paper, as a
result of the author's doctorate program.

\end{document}